\documentclass[12pt,reqno]{amsart}
\usepackage{amsmath,a4wide}
\usepackage{xfrac}
\usepackage{stmaryrd,mathrsfs,bm,amsthm,mathtools,yfonts,amssymb,color}
\usepackage{xcolor}
\usepackage{courier}

\usepackage[final]{hyperref}
\allowdisplaybreaks

\begingroup
\newtheorem{theorem}{Theorem}[section]
\newtheorem{lemma}[theorem]{Lemma}
\newtheorem{proposition}[theorem]{Proposition}
\newtheorem{corollary}[theorem]{Corollary}
\endgroup

\theoremstyle{definition}
\newtheorem{definition}[theorem]{Definition}
\newtheorem{remark}[theorem]{Remark}
\newtheorem{ipotesi}[theorem]{Assumption}

\newtheorem{example}[theorem]{Example}

\numberwithin{equation}{section}
\numberwithin{subsection}{section}

\newcommand{\ph}{\varphi}

\newcommand{\sing}{{\rm Sing}}

\newcommand{\de}{\partial}

\def\a#1{\left\llbracket{#1}\right\rrbracket}

\newcommand{\cH}{{\mathcal{H}}}

\def\I#1{{\mathcal{A}}_{#1}}

\newcommand{\cG}{{\mathcal{G}}}
\newcommand{\etaa}{{\bm{\eta}}}
\newcommand{\D}{\textup{Dir}}

\def\Xint#1{\mathchoice
{\XXint\displaystyle\textstyle{#1}}%
{\XXint\textstyle\scriptstyle{#1}}%
{\XXint\scriptstyle\scriptscriptstyle{#1}}%
{\XXint\scriptscriptstyle\scriptscriptstyle{#1}}%
\!\int}
\def\XXint#1#2#3{{\setbox0=\hbox{$#1{#2#3}{\int}$ }
\vcenter{\hbox{$#2#3$ }}\kern-.6\wd0}}
\def\mint{\Xint-}

\newcommand\N{{\mathbb N}}

\newcommand\C{{\mathbb C}}
\newcommand\R{{\mathbb R}}

\newcommand{\eps}{{\varepsilon}}

\newcommand{\dist}{{\rm {dist}}}

\DeclareMathAlphabet{\mathpzc}{OT1}{pzc}{m}{it}

\title[Number of singular points for planar multivalued harmonic functions]{On the number of singular points for planar multivalued harmonic functions}

\begin{document}

\author[F.~Ghiraldin]{Francesco Ghiraldin}
\address{Max Planck Institut f\"ur Mathematik in den Naturwissenschaften \\
Inselstrasse 22, 04103 Leipzig, Germany}
\email{Francesco.Ghiraldin@mis.mpg.de}
\author[L.~Spolaor]{Luca Spolaor}
\address{Max Planck Institut f\"ur Mathematik in den Naturwissenschaften \\
Inselstrasse 22, 04103 Leipzig, Germany}
\email{Luca.Spolaor@mis.mpg.de}


\date{07 January 2016}

\begin{abstract}In this note we give a quantitative estimate on the number of singular points of multiplicity $Q$ of a $2$-dimensional $Q$-valued energy minimizing map, in terms of the value of its frequency function. 
\end{abstract}

\maketitle

Multivalued functions  have been used by Almgren in \cite{Alm} to study the regularity of Area Minimizing currents of codimension higher than one. 
In particular, the size of the singular set of $\D$-minimizing multivalued functions yields informations on the size of the singular set of Area Minimizing currents (see also \cite{DS2,DS3,DS4,DS5}), and for this reason it might be interesting to have a bound on its measure. 
In \cite{Alm}, Almgren proved that for a $Q$-valued $\D$-minimizing function $f\colon \Omega\subset  \R^m \to \I{Q}(\R^n)$, the dimension of the singular set must be bounded by $(m-2)$, and that it must be locally finite when $m=2$ (see also \cite{DS1}). In \cite{FMS} the authors improved this estimate showing that the Minkowsky dimension is also bounded 
by $(m-2)$.

In this paper, in the case $m=2$, we prove an interior bound on the number of singular points of multiplicity $Q$ of a $\D$-minimizing function in terms of the so-called frequency function. 
This result is optimal, in the sense that  this quantity  cannot bound the number of singular points 
of multiplicity lower than $Q$. The main result is the following (see Section~\ref{s:1} for the precise definitions):
\begin{theorem}\label{t:main}
Suppose that $f\in W^{1,2}(B_2, \I{Q}(\R^n))$ is a  Dir-minimizing function in the ball $B_2\subset \R^2$. 
 Then the following bound holds: 
\begin{equation}\label{e:bound}
\cH^0(\{x\in B_{\sfrac12}\cap \sing(f)\,:\, f(x)=Q\a{p}\})\leq C(n)^{I_{g}(0,2)}\,,
\end{equation}
where $ g= \sum_{i=1}^Q \a{ f_i - \frac 1Q \sum_{i=1}^Q f_i} $ and the frequency function $I_g(x,r)$ is defined by 
\[
I_g(x,r):=\frac{r\,\int_{B_r(x)} |Dg|^2}{\int_{\de B_r(x)}|g |^2}\ \, .
\]
\end{theorem}

The discreteness of the set of singular points of multiplicity $Q$, and indeed of the whole singular set $\sing (f)$, without a quantitative bound on their number, was already proved in \cite{DS1}, and in fact it holds even in the more general context of Area Minimizing $2$-dimensional currents (cf. \cite{Chang, DSS1, DSS2,DSS3,DSS4}). The proof of these results rely on a hard analytical estimate, originally due to Reifenberg (cf. \cite{Reif, Wh, DS1}), called epiperimetric inequality, which implies the uniqueness of the tangent map. We will not need this result in this note and indeed we will only use very soft covering and compactness arguments. However, as a drawback, we will not bound the measure of the whole singular set, but only that of the multiplicity $Q$ points. This is the best one can do, since a bound on the frequency function is in general not enough to control the number of points of multiplicity less than $Q$, as the following example shows. 
\begin{example}
Let $Q=4$ and fix $N$ points $z_i \in B_\frac 12(0)\setminus B_\frac 14(0)$. For every $\eps>0$ let us consider the 
$4$-valued function
 \[
  f_\eps(z) = \pm \left( z (1\pm \eps\prod_{i=1}^{N}(z-z_i)^{\frac 12})\right)^{\frac 12} \in W_{{\rm loc}}^{1,2}(\C,\mathcal{A}_4(\C))\,, 
 \]
 which corresponds to the irreducible algebraic curve 
 \[
 \left\{(z,w)\in \C^2\,:\, (w^2-z)^2=\eps^2\, z^2\, \prod_{i=1}^N(z-z_i) \right\} . 
 \]
 The function $f_\eps$ is $\D$-minimizing, because it is a multivalued holomorphic function, see \cite[Theorem 2.20]{Alm} or \cite{Spada}, and $\sing(f_\eps)=\{0\}\cup\{z_i, i=1,\dots,N\}$. 
For $\eps$ sufficiently small  
 the only $4$-point is the origin, the $\{z_i\}$ are all points of multiplicity $Q'=2$ and $\etaa\circ f_\eps =0$. 
 Letting $\eps\rightarrow 0$ the frequency $I_{f_\eps}(0,2)$ converges to $\sfrac12$, the frequency of the homogeneous 
 limit map $f_0(z)=2\a{\pm z^{\sfrac12}}$. Therefore $N$ can be taken arbitrarily large and $\eps=\eps(N)$ sufficiently small, 
 ruling out the validity of any estimate on the number of $2$-points in terms of $I_{f_\eps}(0,2)$. 
\end{example}

Analogously the number $I(0,2)$ cannot yield any lower bound on the distance between $Q$-points, that is the regularity scale, as the following example shows:
\begin{example}
Consider the $2$-valued function
 \[
  g_\eps(z) = \pm( z (z-\eps))^{\frac 12} \in W_{{\rm loc}}^{1,2}(\C,\mathcal{A}_2(\C))
 \]
 corresponding to the irreducible algebraic  curve
 \[
  \left\{(z,w)\in \C^2\,:   w^2 =  z (z-\eps)      \right\}:
 \]
the points $z=0$ and $z=\eps$ are the only singular points, lying at distance $\eps$, where  $g_\eps(0) = g_\eps(\eps) = 2\a{0}$. Moreover  as before $I_{g_\eps}(0,2)\rightarrow 1$ as $\eps\rightarrow 0$. This example  also indicates that  the proof of Theorem \ref{t:main} is not trivial, as there is no hope to give a lower bound on the scale at which every singular point is isolated.
\end{example}

The optimality of the exponential bound in terms of the frequency is still an open problem. For the singular set of solutions to elliptic PDEs, Lin conjectured the optimal bound to be quadratic in the frequency (see \cite{Lin}). However, Lin's conjecture has been proved only for planar harmonic functions, while in the general case the best possible estimate is exponential (cf. \cite{NV} or \cite{HLBook} and the reference therein for a complete overview). 

This short note is inspired by the papers \cite{Lin,NV}, where the authors study the size of the nodal and critical sets of the solutions of elliptic PDEs, and it can be thought of as a continuation of \cite{FMS}, where the quantitative stratification introduced in \cite{CN1,CN2} is extended to the setting of multivalued maps. Precise references will be given in the rest of the paper.

The note is divided in four short sections: after setting the notations and recalling some preliminary results, we prove bounds on the frequency function in interior balls. 
Then we prove that small frequency drop between two scales implies rules out the existence of any $Q$-point in the corresponding annulus. We finally combine these two ingredients with a covering argument to conclude Theorem \ref{t:main}. 

\noindent
{\bf Acknowledgements:} we wish to thank Emanuele Spadaro for many useful discussions and comments.

\section{Preliminary results and notations}\label{s:1}

Let us start by recalling some known results for $\D$-minimizing $Q$-valued maps. The main references are \cite{DS1, FMS}.

\begin{definition}
We denote the space of $Q$-points by $\I{Q}(\R^n):=\left\{\sum_{i=1}^Q\a{p_i}\}\,:\, p_i\in \R^n\right\}$, 
where $\a{p}$ denotes the Dirac delta at $p$. We can equip it with a complete metric
\[
\cG(T,S):=\min_{\sigma \in \mathcal{P}_Q}\left(  \sum_{i=1}^Q|p_i-p'_{\sigma(i)}|^2  \right)^{\sfrac12}\,,
\]
where $T=\sum \a{p_i},\,S=\sum\a{p'_i}$ and $\mathcal{P}_Q$ is the symmetric group of $Q$ elements. \\
A $Q$-valued map is a measurable function $f\colon \Omega\subset \R^m \to \I{Q}(\R^n)$.
\end{definition}

It is always possible to write almost everywhere $f(x)=\sum_{i=1}^Q\a{f_i(x)}$, where $f_i$ are measurable functions, not necessarily unique.

\begin{definition}[Sobolev spaces] A measurable function $f\colon \Omega \to \I{Q}(\R^n)$ is in the Sobolev class $W^{1,p}$ if there exists functions $\ph_j\in L^p(\Omega, \R^+)$ such that 
\begin{itemize}
\item[(i)] $x\mapsto \cG(f(x),T)\in W^{1,p}(\Omega)$ for all $T\in \I{Q}$;
\item[(ii)] $|\de_j \cG(f,T)|\leq \ph_j$ almost everywhere in $\Omega$ for all $T \in \I{Q}$ and for all $j\in \{1,\dots,m\}$
\end{itemize} 
In particular we denote $|\de_j f|$ the minimal functions satisfying (ii) and we set 
\[
|Df|^2:=\sum_{j=1}^m|\de_jf |^2\,.
\]
Let $\Omega \subset \R^m$ be a Lipschitz bounded open set and $f \in W^{1,p}(\Omega, \I{Q})$. A function $g$ belonging to $L^p(\de\Omega, \I{Q})$ is said to be the trace of $f$ at $\de \Omega$ (and we denote it by $f|_{\de \Omega}$) if, for every $T \in \I{Q}$, the trace of the real-valued Sobolev function $\cG(f,T)$ coincides with $\cG(g,T)$.
\end{definition}

In light of this definitions we can consider the minimization problem  
\begin{equation}\label{e:Dir_min}
\int_{B_2}|Df|^2=\inf\left\{\int_{B_2}|Dh|^2\,:\,h\in W^{1,2}(B_2, \I{Q}(\R^n))\,\text{ and } \,h|_{\de B_2}=f|_{\de B_2}\right\}\,
\end{equation}
and call its solutions Dir-minimizing functions (see \cite[Theorem 0.8]{DS1}). 
Moreover we define $\sing(f)$ as the complement of the set of points 
$x\in B_2$ such that there exists a neighborhood $U$ of $x$ and $Q$ analytic 
functions $f_i\colon U\to \R^n$ satisfying
\[
f(y)=\sum_{i=1}^Q\a{f_i(y)} \,\quad \text{for almost every }y\in B\,,
\]
and either $f_i(x)\neq f_j(x)$ for every $x\in U$ or $f_i\equiv f_j$.

\begin{theorem}[Regularity of $\D$-minimizer {\cite[Theorems 0.9 \& 0.12]{DS1}}]\label{t:regularity}
There exists $\alpha=\alpha(m,Q)\in ]0,1[$ (with $\alpha=\sfrac1Q$ when $m=2$) and $C=C(m,Q,\delta,n)$ such that if $f \in W^{1,2}(B_2,\I{Q})$ is Dir-minimizing, then $f \in C^{0,\alpha}_{{\rm loc}}(B_2, \mathcal A_Q)$ 
and 
\begin{equation}\label{e:C0bound}
[f]_{C^{0,\alpha}(\bar B_\delta)}\leq C \left( \int_{B_2}|Df|^2 \right)^{\sfrac12}\quad\text{for every }0<\delta<2\,.
\end{equation}
Furthermore we have the dimensional bound $\dim_{\cH}(\sing(f)\cap B_2)\leq m-2$. 
\end{theorem}

We remark that in the rest of the work we will need the dimensional bound $\dim_{\cH}(\sing(f)\cap B_2)\leq m-2$ only for homogeneous $\D$-minimizing functions, see the proof of Proposition~\ref{p:freq_drop}.

Finally we introduce the key tool of this work, that is the frequency function.

\begin{definition}[Frequency function] For any $f\in W^{1,2}(\Omega, \I{Q})$, 
$x\in \Omega$ and $r\in (0, \dist(x, \de \Omega))$, such that 
$H_f(x,r)>0$, we define the frequency function 
\[
I_f(x,r):=r D_f(x,r)/H_f(x,r),
\]
where
\[
D_f(x,r):=\int_{B_r(x)} |Df|^2\quad \text{and}\quad H_f(x,r):=\int_{\de B_r(x)}|f|^2\,.
\]
\end{definition}

The following is the main estimate on the frequency function discovered by Almgren \cite{Alm}.

\begin{theorem}[Monotonicity estimate {\cite[Theorem 3.15]{DS1}}]\label{t:mon}Let $f\in W^{1,2}(\Omega, \I{Q})$ be $\D$-minimizing in $\Omega$ and assume that $H_f(x,r)>0$ for every $r\in (s,t)$, with $0\leq s< t< \dist(x, \de \Omega)$. Then $I(x,r)$ is monotone nondecreasing for $r\in (s,t)$ and moreover
\begin{equation}\label{e:monotonicity_identity}
I_f(x,t)-I_f(x,s)=\int_s^t\frac{r}{H_f(x,r)^2}\left(\int_{\de B_r}|\de_r f|^2\,\cdot\int_{\de B_r} |f|^2-\left( \int_{\de B_r}\langle \de_r f, f  \rangle\right)^2\right)\,dr\,.
\end{equation}
\end{theorem}
The frequency function is strongly linked to the growth of the energy and of the $L^2$ norm of a 
$\D$-minimizing map, as recalled in the next proposition, which can be deduced from the first variation formulae and 
Theorem \ref{t:mon}. 
\begin{proposition}[Bounds on Height and Energy {\cite[Theorem 3.15 \& Corollary 3.18]{DS1}}]
Let $f:B_2\rightarrow \mathcal A_Q$ be $\D$-minimizing, 
 $x\in B_1$ and suppose that $H_f(x,r)>0$ for $0<r<\dist(x,\de B_2)$. 
 Then for almost every $r\leq t< \dist(x, \de B_2)$ the following estimates hold 
\begin{gather}
\frac{d}{d\tau}\Big|_{\tau=t} \left[ \ln \left( \frac{H_f(x,\tau)}{\tau^{m-1}}  \right)  \right]=\frac{2\,I_f(x,r)}{r}  \label{e:banana1} \\
\left(\frac{r}{t}\right)^{2 I_f(x,t)} \frac{H_f(x,t)}{t^{m-1}}\leq \frac{H_f(x,r)}{r^{m-1}}\leq \left(\frac{r}{t}\right)^{2 I_f(x,r)} \frac{H_f(x,t)}{t^{m-1}} \label{e:banana2}\\
 \frac{I_f(x,r)}{I_f(x,t)} \left(\frac{r}{t} \right)^{2 I_f(x,t)}\frac{D_f(x,t)}{t^{m-2}}\leq \frac{D_f(x,r)}{r^{m-2}} \leq \left(\frac{r}{t} \right)^{2I_f(x,r)}\frac{D_f(x,t)}{t^{m-2}}\quad \text{provided }I_f(x,r)>0\label{e:banana9}\,.
\end{gather}
\end{proposition}

From now on we will make the following assumptions:
\begin{ipotesi}\label{i:ip}\
\begin{itemize}
\item[(ZM)] $f\in W^{1,2}(B_2,\I{Q})$ is a  $\D$-minimizing function, the baricenter $\etaa\circ f=\frac{1}{Q}\sum_{i=1}^Q f_i=0$ and $H_f(0,2)>0$.
\item[(BF)] $I_f(0,2)\leq \Delta_0<\infty$.
\end{itemize}
\end{ipotesi}
Since the average $\etaa\circ f$ of a $\D$-minimizing function $f$ is a single valued harmonic function, 
subtracting the average preserves the energy minimality \cite[Lemma~3.23]{DS1}.  Moreover if $H_f(0,2)=0$, then by minimality $f\equiv Q\a{0}$ and there would be no singular points: we can therefore assume (ZM) without loss of generality. 
In particular this normalization implies that $f(x)= Q\a{0}$ if and only if $I(x,0^+)>0$. The requirement (BF) is trivial.

We will denote the set of singular $Q$-points by
\[
D_Q:=\{x\in \overline B_{\sfrac12} 	\cap \sing(f)\,:\, f(x)=Q\a{0}\}\,.
\]
\begin{remark}
By unique continuation \cite[Lemma 7.1]{DS4} it is easy to see that 
$D_Q = \{x\in \overline B_{\sfrac12}\,:\, f(x)=Q\a{0}\}$ whenever assumptions \ref{i:ip} hold. 
However for the sake of simplicity we will never use this argument.  
\end{remark}
Finally observe that the definition of regular point includes the case when $f = Q\a{h}$ with 
$h$ an harmonic function. Using the  energy comparison as above it is easy to prove the following lemma:
\begin{lemma}\label{l:patata}
If $x\in D_Q$ then $D_f(x,r)>0$ and $H_f(x,r)>0$ for every $r>0$. In particular the frequency function 
$I_f(x,r)$ is 
well defined. 
\end{lemma}

\section{Interior bounds on the frequency function}
Next we prove that if the frequency is bounded in the ball $B_2$, then it is bounded in all the balls $B_r(x)\subset B_2$. We adapt a result for solutions to elliptic PDEs, which can be found in \cite{Lin}.
  
\begin{proposition}[Frequency bound]\label{p:freq}
Let $f$ be as in Assumptions \ref{i:ip}. There exists a constant $C=C(m)>0$ such that 
\begin{equation}\label{e:frequency_int_bound}
I_f\left(x,1\right)\leq C \, I_f(0,2)\quad \text{for any }x\in B_{\sfrac12}\cap D_Q\,.
\end{equation}
Moreover, for $m=2$, we have $I_f(x,0^+)\geq \tfrac 1Q$ for every $x\in B_{\sfrac12} \cap D_Q$.
\end{proposition}

\begin{proof} 
Since for any $x\in B_{\sfrac12}$ we have $B_{\sfrac32}(x)\subset B_2$ and $B_{\sfrac12}\subset B_{1}(x)$, then by \eqref{e:banana2} with $t=2$ and $r=\sfrac12$, we have
\begin{align}
\mint_{B_{\sfrac32}(x)} |f|^2
&\leq C(m)\,\mint_{B_2} |f|^2 \leq C(m)\,4^{I_f(0,2)}\,\mint_{B_{\sfrac12}}|f|^2 \notag\\
&\leq  C(m)\,4^{I_f(0,2)}\,\mint_{B_{1}(x)}|f|^2\,.\label{e:banana3}
\end{align}
Next we observe that by \eqref{e:banana1} the function $r\to \mint_{\de B_r(x)}|f|^2$ is nondecreasing, so that
\begin{equation}\label{e:banana4}
\int_{B_{\sfrac32}(x)}|f|^2
\geq \int_{\sfrac54}^{\sfrac32} r^{m-1}\,\mint_{\de B_{r}(x)} |f|^2\geq C(m) \mint_{\de B_{\sfrac54}(x)}|f|^2
\end{equation}
and analogously
\begin{equation}\label{e:banana5}
\int_{B_{1}(x)}|f|^2 \leq C(m) \,\mint_{\de B_{1}(x)}|f|^2\,.
\end{equation}
Combining \eqref{e:banana3}, \eqref{e:banana4} and \eqref{e:banana5} we achieve
\begin{equation}\label{e:banana6}
\mint_{\de B_{\sfrac54}(x)}|f|^2\leq C(m)\,4^{I_f(0,2)}\,\mint_{\de B_{1}(x)}|f|^2\,.
\end{equation}
Integrating \eqref{e:banana1}, we get
\[
\ln \mint_{\de B_{\sfrac54}(x)}|f|^2-\ln \mint_{\de B_{1}(x)}|f|^2
=\int_{1}^{\sfrac54} \frac{2 \, I_f(x,r)}{r}\,dr\geq 2\, I_f\left(x,1\right) \ln \left( \frac54\right)
\]
which, together with \eqref{e:banana6}, concludes
\begin{equation}\label{e:banana7}
I_f(x,1)\leq C(m)\, \log\left( C(m)\, 4^{2I_f(0,2)} \right)\leq C(m)+ C(m)\, I_f(0,2)\,.
\end{equation}
To remove the additive constant $C(m)$ and prove \eqref{e:frequency_int_bound}, observe that there exists a constant $\eps=\eps(m)\in (0,1)$ such that if $I_f(0,2)\leq\eps$ then $D_Q\cap B_1=\emptyset$.
Indeed, by \eqref{e:banana2}, if $I_f(0,2)\leq\eps$ then $H_f(0,1)\geq 2^{-m+1-2\eps} H_f(0,2) \geq 4^{-m} H_f(0,2)$ for $\eps$ sufficiently small depending only on $m$. Therefore there exists $x_0\in\partial B_1$ such that $\mathcal G(f(x_0),Q\a{0})\geq 2^{-m}H_f(0,2)^{\sfrac 12}$: by \eqref{e:C0bound} and $D_f(0,2)\leq \eps H_f(0,2)$ we conclude that for every $x\in B_{1}$
\[
\cG(f(x),Q\a{0})\geq \cG(f(x_0),Q\a{0})-\cG(f(x_0),f(x))\geq 2^{-m}H_f(0,2)^{\tfrac 12}- (\eps H_f(0,2)^{\tfrac 12})>0\,.
\] 
In particular $D_Q\cap B_1=\emptyset$. 

To conclude the second claim, observe that by manipulating \eqref{e:banana1} we have
\[
\frac{d}{dr}\left(\frac{H_f(x,r)}{r}\right)=2\frac{D_f(x,r)}{r}.
\]
The H\"older bound \eqref{e:C0bound} and the constraint $f(x)=Q\a{0}$ 
imply that $\lim_{r\rightarrow 0} H(x,r)/r=0$.
Moreover, by energy comparison with an harmonic extension of the boundary datum \cite[Proposition 3.10]{DS1}, and using  the equipartition of energy \cite[Proposition 3.2, (3.6)]{DS1} we have
\[
D_f(x,r)\leq Q\,r\int_{\de B_r(x)}|\de_\tau f|^2 = 
\frac Q2\,r\,D_f'(x,r)
\]
where  $\de_\tau$ denotes the tangential derivative. 
Combining these two inequalities we get
\begin{equation}\label{e:poincare}
\frac{H_f(x,r)}{r} \leq \int_0^r \frac{D_f(x,s)}{s}\, ds\leq Q \, D_f(x,r)\,,
\end{equation}
which yields the desired lower bound. 
\end{proof}

\section{Small frequency drop implies regularity}

In this section we prove that if the frequency drops by a little amount between two scales, then in the corresponding annulus there are no $Q$-points. The idea is that if the frequency drop is small, then the function is $C^{0}$-close to a nontrivial homogeneous $\D$-minimizing function. In the $2$-dimensional case, this blow-up is well separated from $Q\a{0}$ in the annulus, by the characterization of tangent maps \cite[Proposition 5.1]{DS1}.

\begin{definition} A $Q$-valued function $w\in W^{1,2}_{{\rm loc}}(\R^m,\I{Q})$ is called $\alpha$-homogeneous if
\[
w(x)=|x|^\alpha w\left(\frac{x}{|x|} \right) \quad \forall x\in \R^n\setminus \{0\}.
\]
We will denote by $\cH_{\Delta_0}$ the class of $\alpha$-homogeneous, 
locally $\D$-minimizing functions $w$, with $\alpha \leq \Delta_0$, 
$D_w(0,1) = 1$
 and $\etaa\circ w=0$. 
\end{definition}
\begin{remark}\label{re:homog}
Recall that for a given $\D$-minimizing function $f$, the following conditions are equivalent, see \cite[Corollary 3.16]{DS1}:
\begin{itemize}
\item[(i)] $f$ is $\alpha$-homogeneous;
\item[(ii)] $I_f(0,r)=\alpha$ for every $r>0$;
\item[(iii)] the monotonicity remainder satisfies
\[
\int_0^1\frac{r}{H_f(r)^2}\left(\int_{\de B_r}|\de_r f|^2\,\int_{\de B_r} |f|^2- \left(\int_{\de B_r}\langle \de_r f, f  \rangle\right)^2\right)\,dr=0.
\]
\end{itemize}   
\end{remark}

We define the rescaled maps
\begin{equation}\label{eq:rescal}
f_{x,s}(y):=\frac{s^{\frac{m-2}{2}}f(x+s y)}{D_f^{\sfrac12}(x, s)}\colon B_1\to \I{Q}
\end{equation}
for every $x\in D_Q$. Thanks to assumption (ZM), these rescalings are well defined and satisfy $D_{f_{x,s}}(0,1)=1$. 

\begin{proposition}[Frequency drop, see {\cite[Lemma 4.0.1]{FMS}}]\label{p:freq_drop} Let $f$ be as in Assumptions \ref{i:ip}. For every $\eps>0$ there exist $\delta,\lambda \in ]0,\sfrac15[$  
such that for every $x\in B_1\cap D_Q$ and $r\in ]0,\sfrac 12\dist (x, \de B_2)[$ the following implication holds
\begin{equation}\label{e:fd1}
I_f(x,r)-I_f(x, \lambda r)\leq \delta \quad \Rightarrow \quad \exists w\in \cH_{\Delta_0}\text{ s.t. } \|\cG(f_{x,r},w)\|_{C^0(B_1\setminus B_{\lambda})}\leq \eps\,.
\end{equation}
\end{proposition}

\begin{proof} 
Suppose the statement is not true, then there exist sequences of points $(x_j)_j\subset B_1$ 
and of radii $r_j<\sfrac 12$ such that
\begin{equation}\label{e:pinch0}
I_f(x_j,r_j)-I_f(x_j, 2^{-j} r_j)\leq 2^{-j} \quad \text{and} \quad \|\cG(f_{x_j,r_j},w)\|_{C^0(B_1\setminus B_{\lambda})}\geq \eps\quad \forall w\in \cH_{\Delta_0}.
\end{equation}
For simplicity, set $f_j:=f_{x_j,r_j}\in W^{1,2}(B_2,\mathcal A_Q)$ and denote by $I_j:=I_{f_j}$. 
Observe that $f_j(0)=Q\a{0}$, $D_j(0,1) = 1$ and that \eqref{e:pinch0} becomes
\begin{equation}\label{e:pinch1}
I_j(0,1)-I_j(0,2^{-j})\leq 2^{-j}\quad \text{and}\quad \|\cG(f_j,w)\|_{C^0( B_1\setminus B_{\sfrac15})}\geq \eps\quad \forall w\in \cH_{\Delta_0}\,.
\end{equation}
Moreover observe that, by Proposition~\ref{p:freq}, $\sup_jI_j(0,2)\leq C(m)\, \Delta_0$. 
Since by \eqref{e:banana1} the function $r\to H_j(0,r)$ is increasing, we have
\[
\int_{B_2}|f_j|^2\leq2 H_j(0,2)\leq 2^{2+2I_j(0,2)} H_{j}(0,1) = \frac{2^{2+2I_j(0,2)}}{I_j(0,1)}\leq Q2^{2+2I_j(0,2)}\leq C(Q,\Delta_0)\,,
\]
where we have used \eqref{e:banana2} in the second inequality and the lower bound $I_j(0,1)\geq \tfrac 1Q$, of Proposition~\ref{p:freq}, in the fourth inequality. 
Combining this with \eqref{e:banana9} and using again Proposition~\ref{p:freq}, we achieve
\[
\sup_{j}D_j(0,2)\leq 2^{1+2I_j(0,2)}\frac{I_j(0,2)}{I_j(0,1)}\leq C(Q,\Delta_0)\,.
\]
Therefore $(f_j)$ is equibounded in $W^{1,2}(B_2, \I{Q})$ 
and by Theorem \eqref{t:regularity}, precompact in $C^0_{{\rm loc}}(B_2, \mathcal A_Q)$. 
Up to the extraction of a subsequence, $f_j\to f$ weakly in $W^{1,2}(B_2,\I{Q})$ and strongly in $C^0_{{\rm loc}}(B_2,\I{Q})$, where $f\in W^{1,2}(B_2, \I{Q})$. 
Moreover by \cite[Proposition 3.20]{DS1} $f$ is $\D$-minimizing and  the following strong convergence holds: 
\begin{equation}\label{e:D-conv}
\lim_{j\to \infty}D_{f_j}(x,r)= D_f(x,r),\quad \forall x\in B_2, \,0<2r<\dist(x, \de B_2)\,:
\end{equation}
this in particular ensures that $D_f(0,1)=1$. Clearly also $f(0)=Q\a{0}$ holds and by locally uniform convergence 
\begin{equation}\label{e:I-conv}
\lim_{j\to \infty} I_{f_j}(x,r)=I_f(x,r) ,\quad \forall x\in B_2, \,0<2r<\dist(x, \de B_2)\,.
\end{equation}
Using \eqref{e:D-conv} with $x=0$, \eqref{e:monotonicity_identity} and \eqref{e:pinch1} we can pass to the limit and deduce 
\begin{equation}\label{e:pinch2}
\int_0^1\frac{r}{H_f(r)^2}\left(\int_{\de B_r}|\de_r f|^2\,\cdot\int_{\de B_r} |f|^2- \left(\int_{\de B_r}\langle \de_r f, f  \rangle\right)^2\right)\,dr=0\,.
\end{equation}
By Remark \ref{re:homog}, equation \eqref{e:pinch2} implies that $f$ is $\alpha$-homogeneous; moreover, by \eqref{e:D-conv} we have $D_f(0,1)=1$ and by \eqref{e:I-conv} we conclude $\alpha = I_f(0,1)\leq \Delta_0$. Since $f_j$ converges to $f$ uniformly in $B_1$ we have reached a contradiction in \eqref{e:pinch0} with $f=w$.
\end{proof}

\begin{corollary}\label{c:annulus-regularity} Let $f$ be as in Assumptions \ref{i:ip} with $m=2$. There exist $\delta,\lambda \in ]0,\sfrac15[$ such that
\begin{equation}
x\in D_Q,\qquad I_f(x,r)-I_f(x, \lambda r)\leq \delta \quad \Rightarrow \quad D_Q\cap (B_r(x)\setminus B_{\lambda r}(x))=\emptyset\,.
\end{equation}
In particular, $D_Q$ is locally finite.
\end{corollary}

\begin{proof} Reasoning as in the proof of Proposition~\ref{p:freq_drop}, $\cH_{\Delta_0} $ is compact in $C^0(\overline{B_1},\mathcal A_Q)$, hence there exists $\eta>0$ such that 
\begin{equation}\label{e:low_bound}
\inf_{y\in B_1\setminus B_{\sfrac12}}|w(y)| \geq \eta>0\quad \forall w\in \cH_{\Delta_0} \,.
\end{equation}
Indeed if this were not the case, one could find an $\alpha$-homogeneous function $w\in \cH_{\Delta_0}$ 
and $y\in\overline{B}_1$ such that $w(ry)=0$ for every $r>0$, which is a contradiction with the dimensional estimate $\dim_{\cH}(\sing(f))=0$ in Theorem~\ref{t:regularity}.

Now choose $\eps \geq \sfrac{\eta}{2}$, and let $w\in  \cH_{\Delta_0}$, $\delta>0$ and $ \lambda<\sfrac15$ be given by Proposition \ref{p:freq_drop} for our choice of $\eps$ and $x$. 
For every $y\in B_r(x)\setminus B_{\lambda r}(x)$, writing $y=x+\rho z$, where $\rho= |y|\in ]\lambda r, r[$, by a simple triangular inequality we have
\begin{align*}
|f(y)|
&=D_f^{\sfrac12}(x,\rho) \, |f_{x,\rho}(z)|\geq D_f^{\sfrac12}(x,\rho) \Bigl( |w(z)|-\cG(f_{x,\rho}(z), w(z)) \Bigr)\\
& \geq D_f^{\sfrac12}(x,\rho)\,\Bigl(\eta-\eps \Bigr)\geq D_f^{\sfrac12}(x,\rho)\frac{\eta}{2}>0\,,
\end{align*}
so that $y\notin D_Q$. Notice that $D_f^{\sfrac12}(x,\rho)>0$ for every $\rho= |y|\in ]\lambda r, r[$, because otherwise, since $\etaa\circ f\equiv 0$, we would have $f\equiv Q\a{0}$ in $B_{\lambda r}(x)$, 
which contradicts Lemma~\ref{l:patata}. 

To prove the last part of the statement, assume it is not true, then there is a sequence $(x_j)_j\subset D_Q$ converging to a point $x\in \overline B_{\sfrac12}$. Since $f$ is continuous, $x\in D_Q$, and the sequence $r_j:=2|x-x_j|$ satisfies
\[
 x_j\in D_Q\cap (B_{r_j}(x)\setminus B_{\lambda r_j}(x))\quad\forall j\in \N.
\]
By \eqref{e:frequency_int_bound} we can bound
\[
 \sum_j \left(I_f(x,r_j)-I_f(x, \lambda r_j)\right)\leq C(\lambda) I_f(x,1)\leq C(m)\,\Delta_0;
\]
on the other hand by the first part of the Corollary the sum on the left hand side is infinite, 
which gives a contradiction.
\end{proof}

\section{Covering argument and the Proof of Theorem \ref{t:main}}

The following covering argument is adapted from \cite{NV}.

\begin{proof}[Proof of Theorem \ref{t:main}.] We set $N_0=\cH^0(D_Q)$ and 
 we proceed inductively as follows: at the initial step 
 we cover the set $D_Q$ with the collection of balls $\{B_{\lambda}(x)\}_{x\in D_Q}$, 
 where $\lambda$ is as in Corollary \ref{c:annulus-regularity}. 
 From this collection we extract a  Vitali subcover $\{B_\lambda(x_j)\}_{j=1}^{J(0)}$, that is
\[
D_Q\subset \bigcup_{j=1}^{J(0)}B_{\lambda}(x_j)\subset B_1
\quad \text{and} \quad
B_{\lambda^2}(x_j)\cap B_{\lambda^2}(x_i)=\emptyset\,\text{ whenever }j\neq i\,.
\] 
Choose $x^1\in\{x_j\}_{j=1}^{J(0)}$  such that
\[
N_1:=\cH^0(D_Q\cap B_{\lambda}(x^1))\geq \cH^0(D_Q\cap B_{\lambda}(x_j))\quad \forall j=1,\dots,J(0)
\]
and consider the ball $B_{\lambda}(x^1)$. At the $k^{th}$ step we are given a ball $B_{\lambda^k}(x^k)$ and we cover it as above with balls $\{B_{\lambda^{k+1}}(x^{k+1}_j)\}_{j=1}^{J(k)}$ such that $x^{k+1}_j\in D_Q\cap B_{\lambda^k}(x^k)$ and 
\[
D_Q\cap B_{\lambda^k}(x^k)\subset \bigcup_{j=1}^{J(k)}B_{\lambda^{k+1}}(x^{k+1}_j)\subset B_{2\lambda^k}(x^k_1)
\quad \text{and} \quad
B_{\lambda^{k+1}}(x^{k+1}_j)\cap B_{\lambda^{k+1}}(x^{k+1}_i)=\emptyset\,\text{ if }j\neq i\,.
\] 
Moreover choose $x^{k+1}\in\{x^{k+1}_j\}_{j=1}^{J(k)}$ is such that
\[
N_{k+1}:=\cH^0(D_Q\cap B_{\lambda^{k+1}}(x^{k+1}))\geq \cH^0(D_Q\cap B_{\lambda^{k+1}}(x^{k+1}_j))\quad \forall j=1,\dots,J(k)\,.
\]
Observe that at each step $k$ the number of balls $J(k)$ in the cover is bounded independently of $k$ by
\begin{equation}\label{e:N_balls}
J(k)\leq \frac{|B_{2\lambda^k}|}{|B_{\lambda^{k+1}}|}\leq \frac{4}{\lambda^2}\,.
\end{equation}
Next we define a function $\xi\colon \N\setminus \{0\}\to \{0,1\}$ by 
\[
\xi(k):= 
\begin{cases} 
1 &\mbox{if } N_{k-1}>N_k \\
0 &\mbox{if } N_{k-1}=N_k \,.
\end{cases}
\]
Since by Corollary \ref{c:annulus-regularity} $D_Q$ is locally finite, there exists $\bar k=\bar k_f\in \N$ such that $\xi(k)=0$ for every $k>\bar k$, that is $B_{\lambda^{\bar k}}(x^{\bar{k}})\cap D_Q= \{x^{\bar{k}}\}$ and $N_{\bar k}=1$.\\ 
Notice also that, by definition of $B_{\lambda^k}(x^k)$ as the ball with the highest number of $Q$-points in the cover, it follows that $N_k \geq \frac{N_{k-1}}{J(k-1)}$, and so, using this when $\xi(k)=1$ and $N_{k-1}=N_k$ when $\xi(k)=0$, we conclude
\begin{equation}\label{e:bound_1}
N_0\leq \left(\sup_k J(k)\right)^{\sum_k \xi(k)} N_{\bar{k}}\leq \left(\frac{4}{\lambda^2}\right)^{\sum_k \xi(k)}\,.
\end{equation}

To conclude we claim that
\begin{equation}
\sum_k \xi(k)\leq C\, I_f(x_{\bar{k}},1)\stackrel{\eqref{e:frequency_int_bound}}{\leq} C\,I_f(0,2) \,\,. 
\end{equation}
which, combined with \eqref{e:bound_1}, proves Theorem \ref{t:main}.
We only need to consider the cases when $\xi(k)=1$, that is $N_{k-1}>N_k$. In this situation there exists a point 
\[
y^k\in \left(B_{\lambda^{k-1}}(x^{k-1})\cap D_Q\right)\setminus B_{\lambda^{k}}(x^{k}),
\] which in particular satisfies $\lambda^{k}\leq |x^k-y^k|\leq 2 \lambda^{k-1}$. 
Consider the ball $B_{\lambda^{k+1}}(x^{k+1})$ and observe 
that if there exists $y\in B_{\lambda^{k+1}}(x^{k+1})$ such that 
$|y-y^k|\leq \lambda^{k+1}$, then 
$|x^k-z|\geq |x^k-y^k|-2 \lambda^{k+1}\geq \lambda^{k+1}$ for every $z\in B_{\lambda^{k+1}}(x^{k+1})$. Therefore we can choose $z^k\in D_Q$ equal either to $x^k$ or $y^k$ such that 
\begin{equation}\label{e:basta!}
\lambda^{k+1} \leq |z^k-z| \leq 3\lambda ^{k-1} \quad \forall z\in B_{\lambda^{k+1}}(x^{k+1})\,. 
\end{equation}
where the second inequality follows from the fact that $z^k, x^{k+1}\in B_{\lambda^{k-1}}(x^{k-1})$. That is, since $x^{\bar k}\in B_{\lambda^k}(x^k)$ for every $k\in \N$, choosing $z=x^{\bar k}$ in \eqref{e:basta!} we have
\[
\xi(k)=1 \quad \Rightarrow\quad \exists \,z^k\in D_Q\cap \Bigl(B_{3\lambda^{k-1}}(x^{\bar k})\setminus B_{\lambda^{k+1}}(x^{\bar k})\Bigr).
\] 
By Corollary \ref{c:annulus-regularity}, this implies that 
\[
  \delta\, \left( \sum_k \xi(k)\right) \leq \sum_{\{k\,:\, \xi(k)=1\}} \bigl(I_f(x^{\bar k},3\lambda^{k-1})-I_f(x^{\bar k},\lambda^{k+1})\bigr)  \leq C(\lambda)\,I_{f}(x^{\bar{k}}, 1)
\] 
which proves the claim.
\end{proof}

\bibliographystyle{plain}
\bibliography{references-Cal}

\end{document}